\documentclass[twoside]{article}
\usepackage{amsmath,amssymb,amscd,amsthm,verbatim,alltt,amsfonts,array,mathtools,bbm}
\usepackage{mathrsfs}
\usepackage[english]{babel}
\usepackage{latexsym}
\usepackage{amssymb}
\usepackage{euscript}
\usepackage{graphicx}

\newcommand{\C}{{\mathbb C}}

\newcommand{\ko}{{\mathcal O}}

\newtheorem{theorem}{Theorem}

\newtheorem{conjecture}[theorem]{Conjecture}
\newtheorem{proposition}[theorem]{Proposition}

\begin{document}

\begin {center}
{\bf Appendix: History of Singular\\
 and its relation to Zariski's multiplicty conjecture}\\
Gert-Martin Greuel and Gerhard Pfister
\end  {center}
\bigskip

When you call {\sc Singular}, local on your computer or online, the following heading appears:
\begin {center}
SINGULAR\\                                
 A Computer Algebra System for Polynomial Computations\\       
 by: W. Decker, G.-M. Greuel, G. Pfister, H. Schoenemann\\  
 FB Mathematik der Universitaet, D-67653 Kaiserslautern
\end{center}

In fact, {\sc Singular} is nowadays a widely used a computer algebra system for polynomial computations with special emphasis on the needs of commutative algebra, algebraic geometry, and singularity theory. However, at the beginning this was never planned, we just wanted to solve mathematical problems. Only later when we had been (partially) successful, we decided to create a system also to be used by others. The development of  {\sc Singular} has been strongly motivated and was for a long period mainly driven by mathematical problems in singularity theory. Even its appreciated computational speed is a consequence of singularity problems which are theoretically as well as computationally very hard. It is perhaps of interest to the singularities community to see how it all came about.

It started at a time, when symbolic computations was just beginning to emerge and algorithms, in particular for local computations, were practically not existent. Moreover,  our cooperation within two Germanies was anything but easy because a visit from East Germany to West Germany was not possible. Anyway, we could meet in East Berlin and we started a cooperation around 1984. 

\begin{center}\includegraphics[width=6cm]{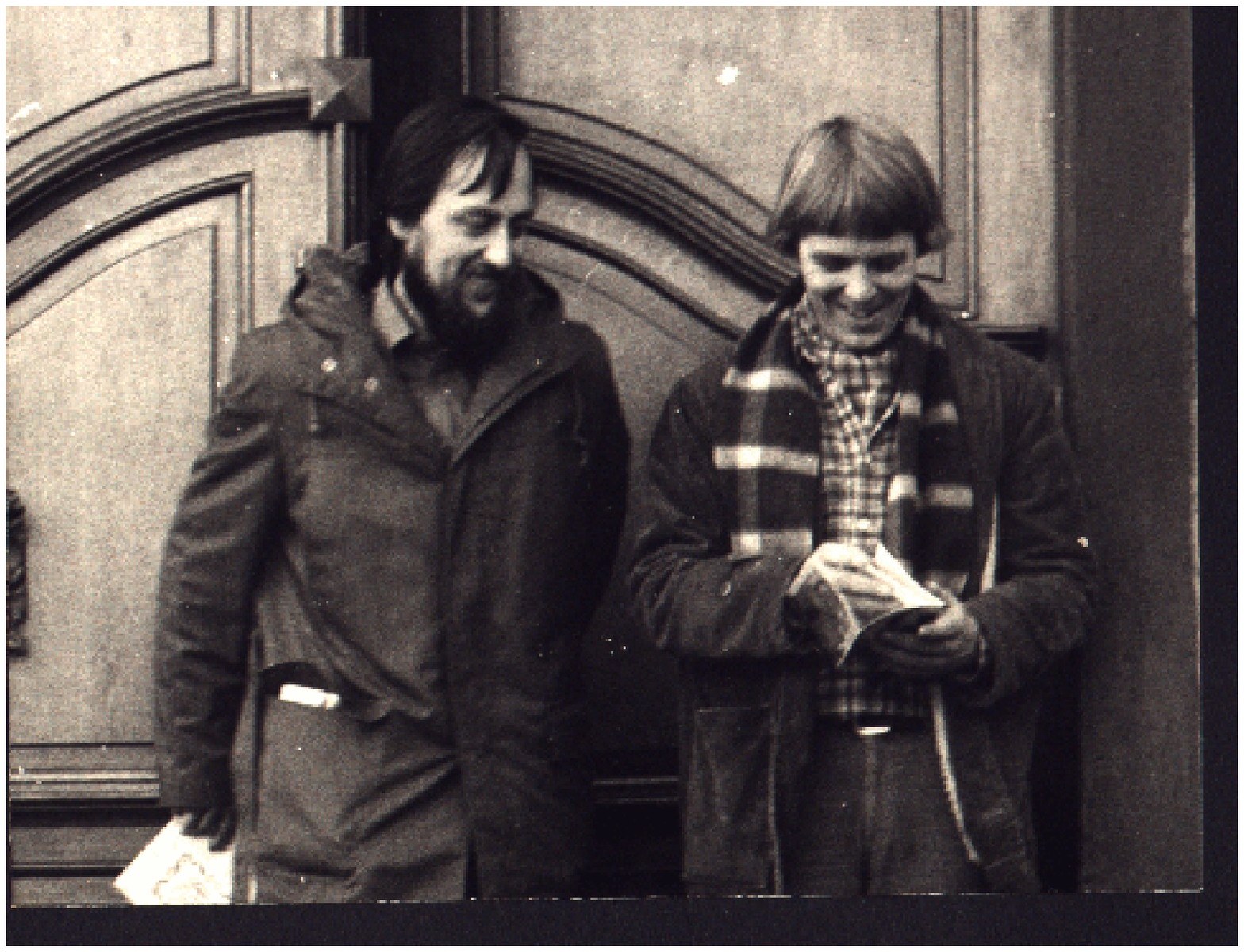}\\
Pfister (left) and Greuel at Humboldt University Berlin, 1984
\end{center}

%
The birth of {\sc Singular} goes back to our efforts to generalize Kyoji
Saito's well known result for hypersurface singularities (cf. \cite{S}):

\begin{theorem}[K. Saito, 1971]
Let $(X,0)$ be the germ of an
isolated complex hypersurface singularity. The following conditions are equivalent:
\begin{enumerate}
\item $(X,0)$ is quasi-homogeneous (that is, has a good
  $\mathbb C^\ast$-action). 
\item $\mu (X,0) = \tau (X,0) $.
\item The Poincar\'e complex of $(X,0)$ is exact.
\end{enumerate}
\end {theorem}

If $(X,0)$ is given by $f\in \C\{x_1,...,x_n\}$ then 
$\mu (X,0) = \dim_\C \C\{x_1,...,x_n\}/ j(f)$
is the Milnor number and $\tau (X,0)=\dim_\C \C\{x_1,...,x_n\}/ \langle f, j(f)\rangle$ the Tjurina number, with 
$j(f)=\langle {\partial f}/{\partial x_1}, \ldots, {\partial f}/{\partial x_n}\rangle$. 

Based on results in  \cite{G1} and \cite{G} we proved in \cite{GMP} the following generalization of Saito's result to 
isolated complete intersection curve singularities (the result  in \cite{GMP}  was more general for reduced Gorenstein curve singularities, with the Tjurina number replaced by the Deligne number).

\begin{theorem}[G.-M. Greuel, B. Martin, G. Pfister, 1985] If $(X,0)$ is a reduced complete intersection curve
singularity, then 
$$ (X,0) \text{ quasi-homogeneous } \Longleftrightarrow \; \mu (X,0)
= \tau (X,0)\,. $$ 
\end{theorem}

So we asked ourselves in \cite[Problem 1]{GMP} whether $(X,0)$ is quasi-homogeneous if the Poincar\'e complex of $(X,0)$ is exact (the other direction is clear). At the beginning we actually conjectured that the answer should be positive.  However, we did not succeed in proving it and so we started to look for possible counter examples. But the computations by hand were very time consuming and with the small examples at hand we were unable to find any counter example. Nevertheless, we started not to believe in the conjecture.
 
To compute potential counter examples with a help of a computer, two main problems appeared: First, we needed Teo Mora's tangent cone algorithm (a variation of Buchberger's algorithm for local rings) to compute standard bases  for 
$\ko_{X,0}$-modules. However, no package for this existed at that time, not even for ideals.
The second problem was more of  a theoretical nature. We needed to compute the kernel of the exterior derivation in the Poincar\'e complex, which is only $\C$-linear but not $\ko_X$-linear and hence not directly  tractable by standard bases computations. Fortunately, using a result of Reiffen (see below)  we had been able 
 in  \cite{GMP} to reformulate the exactness of the 
Poincar\'e complex as a question of computing submodule membership and  dimensions of $\ko_{X,0}$-modules.

The first problem was more serious. There was no computer algebra system available which could compute this kind of examples. In 1984 Neuendorf and Pfister (during vacations at the baltic sea) started an implementation of Buchberger's Gr\"obner 
basis algorithm in Basic on a ZX-Spectrum (an 8 bit home PC from Sinclair UK, 1982).  It took Pfister and his student Hans Sch\"onemann two more years of
development to obtain a Modula-2 implementation of a package, called {\em Buchmora} at that time (Buchberger's and Mora's algorithm) for Atari computers.
Using this implementation the following counter examples were found (cf. \cite{PS}).

\begin{theorem}[G. Pfister, H. Sch\"onemann, 1989]
 Let $(X_{lk},0)$ be the germ of the unimodal space curve singularity 
$FT_{k,l}$ of the
classification of C.T.C. Wall (cf. \cite{W}) defined by the equations
$$
xy + z^{l-1} = xz + yz^2 + y^{k-1} = 0\,, \quad 
(4 \leq l \leq k ,\:  5 \leq k)\,.
$$
Then the Poincar\'e complex
$$0 \longrightarrow \mathbb C \longrightarrow \mathcal O_{X_{lk},0} 
\longrightarrow
\Omega^1_{X_{lk},0} \longrightarrow \Omega^2_{X_{lk},0} \longrightarrow
\Omega^3_{X_{lk},0} \longrightarrow 0$$
is exact,
but $(X_{lk},0)$ is not quasi-homogeneous.
\end{theorem}

\begin{proof} To show that $(X_{lk},0)$ is not quasi-homogeneous, it suffices to show
$$ \mu (X_{lk},0) = \tau (X_{lk},0) + 1 = k + l + 2$$
by the following formulas.
Let \mbox{$(X,0)\subset (\mathbb C^3\!,0)$ }be the space curve singularity defined by
\mbox{$f = g = 0$}, with \mbox{$f,g \in \mathbb C \{x,y,z\}$}. Then
\begin{itemize}
\item
$\mu (X,0) =
  \dim_{\mathbb C}(\Omega^1_{X,0}/d\mathcal{O}_{X,0})$
$ = \dim_{\mathbb C}\mathbb C \{x,y,z\}/\langle
f,M_1,M_2,M_3\rangle$\\
\hspace*{5.2cm} $-\dim_{\mathbb C} \mathbb C \{x,y,z\}/\langle
\frac{\partial f}{\partial x},\frac{\partial f}{\partial y},\frac{\partial
  f}{\partial z}\rangle$,
\item

$\tau (X,0) = \dim_{\mathbb C}\mathbb C \{x,y,z\}/\langle
f,g,M_1,M_2,M_3\rangle$,
\end{itemize}

\noindent
with $M_1,M_2,M_3$ the 2-minors of the Jacobian matrix 
$\left(
\begin{matrix}
\frac{\partial f}{\partial x}&\frac{\partial f}{\partial y} &
\frac{\partial f}{\partial z}\\
\frac{\partial g}{\partial x}&\frac{\partial g}{\partial y} &
\frac{\partial g}{\partial z}
\end{matrix}\right)
$.

\noindent
On the other hand, a result of Reiffen says:
\begin{itemize}
\item  The Poincar\'e complex is exact iff
\begin{quote}
\begin{enumerate}
\item 
$\langle f,g\rangle \cdot \Omega^3_{\mathbb C^3,0} \subset d(\langle
f,g\rangle\cdot \Omega^2_{\mathbb C^3\!,0})$, and
\item 
$\mu (X,0) = \dim_{\mathbb C}(\Omega^2_{X,0}) - \dim_{\mathbb C}(\Omega^3_{X,0})$.
\end{enumerate}
\end{quote}
\end{itemize}
All these statements could be checked with the Buchmora algorithm.
\end{proof}

Encouraged by this success and having a computer algebra system that was able to compute in local rings, we tried to find a counter example to  Zariski's multiplicity conjecture (Zariski had posed this as a question, which he supposed to have quick answer by topologists, cf. \cite{Z}).

\begin{conjecture}[O. Zariski, 1971]  
Two hypersurface singularities (given by convergent power series) with the same topological type have the same multiplicity.
\end{conjecture}

A weaker version of this conjecture says:\\
{\em In a $\mu$--constant deformation of an isolated 
hypersurface singularity the  multiplicity is constant.} 
\medskip

The conjecture was already known for reduced plane curve singularities and the weaker conjecture for isolated quasi-homogeneous hypersurface singularities (cf. \cite{G2}). The methods of  \cite{G2} are in principal applicable to any isolated hypersurface singularity, but we failed to prove the weak Zariski's conjecture in general. Due to the many unsuccessful efforts by us and others  we were (and are still) convinced that Zariski's conjecture might not be true. 
\medskip

Hence, we tried to find a counterexample. The main problem is the difficulty to construct examples of $\mu$-constant deformations. Since Zariski's conjecture is true in the semi quasi-homogeneous case and for plane curve singularities, the examples to test should be somewhat complicated.
We used the Newton diagram to construct families of surface singularities where the multiplicity drops and with Newton diagram becoming degenerate but with rather small degeneracy area, hoping that the Milnor number would stay constant. 
Among others we tried a series of examples of the following form:

\begin{center}\includegraphics[width=6cm]{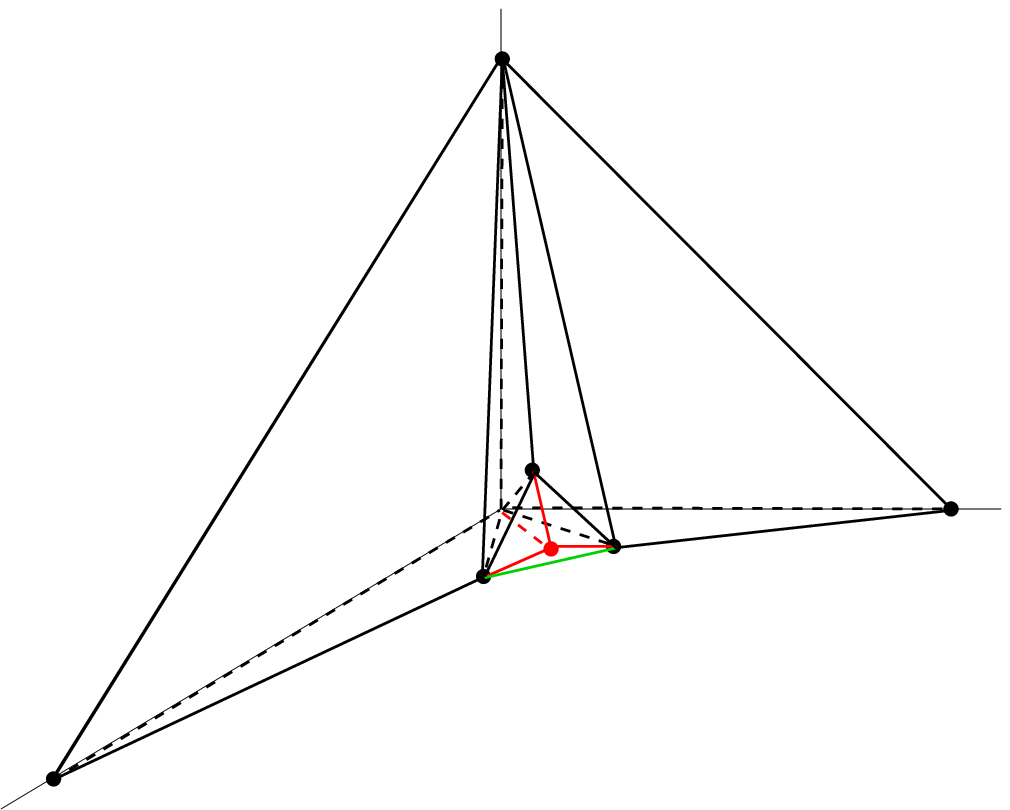}\\
$F_t = x^a + y^b + z^{3c} + x^{c+2}y^{c-1} + x^{c-1}y^{c-1}z^3 + x^{c-2}y^c(y^2 + tx)^2$
\end{center}

The multiplicity can be read of from the equation, but for the Milnor number we had to use a computer and the package Buchmora. However, this and other examples took hours to compute.  Whenever we met, in East Germany (often in Pfister's dacha close to Berlin) or at conferences outside West Germany, we tried to improve the algorithm by checking different local orderings and trying to
optimize the selection strategies during the standard basis computation (producing huge tables with timings). The selection strategies for different orderings, which we finally preferred, are still in use in the present version of  {\sc Singular}.

\begin{center} 
\includegraphics[width=6cm]{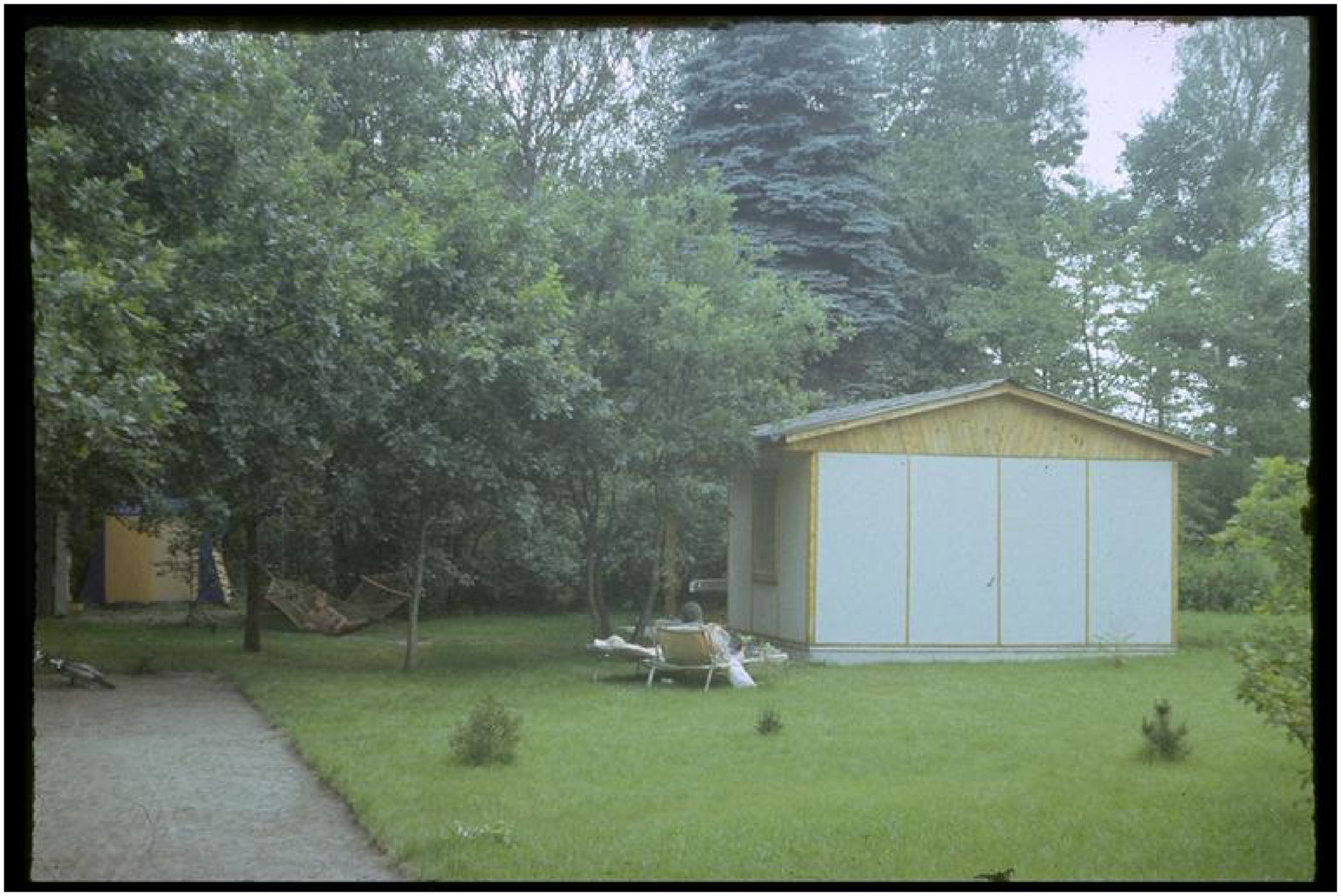}\\
The place where everything started: Pfister's dacha in the GDR
\end{center}

Among the above series of examples there was unfortunately no counter example. We found e.g. for $(a,b,c) = (40,30,8)$:
$$m(F_0) = 17, \ m(F_t) = 16, \ \mu(F_0) = 10661, \ \mu(F_t) = 10655.$$ 
The computations for $\mu$ took many hours (today within a few seconds), but smaller Milnor numbers could be excluded by heuristical arguments. A significant speed up of the computation of standard bases for local orderings was needed
and we decided to make a further step towards  a more professional development of a computer algebra package.\\

In 1989 Buchmora was renamed to {\sc Singular}. It was jointly developed by groups from Berlin (Pfister) and Kaiserslautern (Greuel) within a  DFG priority program 1990--1996.  Within this program we could hire Hans Sch\"onemann, who moved to Kaiserslautern in 1990, right after the unification of Germany.
 {\sc Singular} was ported to Unix (still in Modula-2) and a first user manual was released.  
In 1993 Pfister moved to Kaiserslautern and we decided to rewrite the code in C/C++, carried out mainly by Sch\"onemann. Within the  DFG priority program the {\sc Singular} programming language was developed and many  libraries had been established. Around 1996
Olaf Bachmann joined the team in Kaiserslautern and with his help it was possible to improve the code of {\sc Singular}  significantly, mainly by adapting the data structures and the memory management, which increased the speed drastically.

In spite of these improvements, no counter example was found! But by analyzing the above examples  a partial solution to Zariski's conjecture was published in \cite{GP}, including the
first publication of a standard basis algorithm for arbitrary mixed monomial orderings (implemented in Singular since 1993):

\begin{proposition}[G.-M. Greuel, G. Pfister, 1996]: 
Let 
$$F_t(x_1,\ldots,x_n)=G_t(x_1,\ldots,x_{n-1}) +x_n^2H_t(x_1,\ldots,x_n)$$ be a family of isolated hypersurface singularities.
Let $G_0$ be semiquasihomogeneous or let $n=3$.  If the family has constant Milnor number and the multiplicity of $G_t$ is smaller
or equal to the multiplicity of $H_t+2$ then the multiplicity of $F_t$ is constant.
\end{proposition}

To conclude, let us remark, that the {\em failure} to find a counter example to Zariski's conjecture was the most important reason for the development of {\sc Singular} as it is now. First of all, for many years it was the main motivation to improve its speed, since the possible counter examples were complicated to compute.
Secondly, it was a very good theoretical problem that convinced the referees  to support the development of {\sc Singular} for many years.

\bigskip
{\sc Singular} -- Some History
\begin{itemize}
\item 1984  Neuendorf/Pfister: Implementation of the Gr\"obner 
basis algorithm in Basic on a ZX-Spectrum. 
\item 1990 Sch\"onemeann moved to KL, porting to Unix
\item 1993 Pfister moved to KL, C/C++ version. 
\item 1996--2000 Greuel/Pfister:  symbolic/numerical algorithms in {\sc Singular}, joint with electrical engineers and a Mathematic package ''Analog Insydes''.
\item 1997/1998 Singular release 1.0 -1.2, with multivariate polynomial factorization, gcd, syzygies, free resolutions, communication links, primary decomposition and normalization.
\item 2002
 Book:  A SINGULAR  Introduction to Commutative Algebra.\\ By G.-M. Greuel and G. Pfister, with contributions by O. Bachmann, C. Lossen and H. Sch\"onemann.
\item 2004 First
Richard D. Jenks Memorial Prize
 for Excellence in Software Engineering
awarded to {\sc Singular}  at ISSAC in Santander.
\item 2004 Greuel/Levandovsky: 
The subsystem PLURAL for non-commutative polynomial algebras is included in {\sc Singular}.
\item 2008 interface to the computer algebra system ''Sage''.
\item 2009 Decker moves to KL, with Greuel/Pfister/Sch\"onemann one of the leaders of the  {\sc Singular} developmnet.
\item 2016 The ''Oscar'' system includes a Julia package for the Singular library. 
\item  {\sc Singular}  has been supported by Deutsche Forschungsgemeinschaft (DFG), Stiftung Rheinland-Pfalz f\"ur Innovation, and Volkswagen Stiftung.
\item  {\sc Singular}    is free software, available at  https://www.singular.uni-kl.de/
\end{itemize}

\begin{center} 
\includegraphics[width=3in]{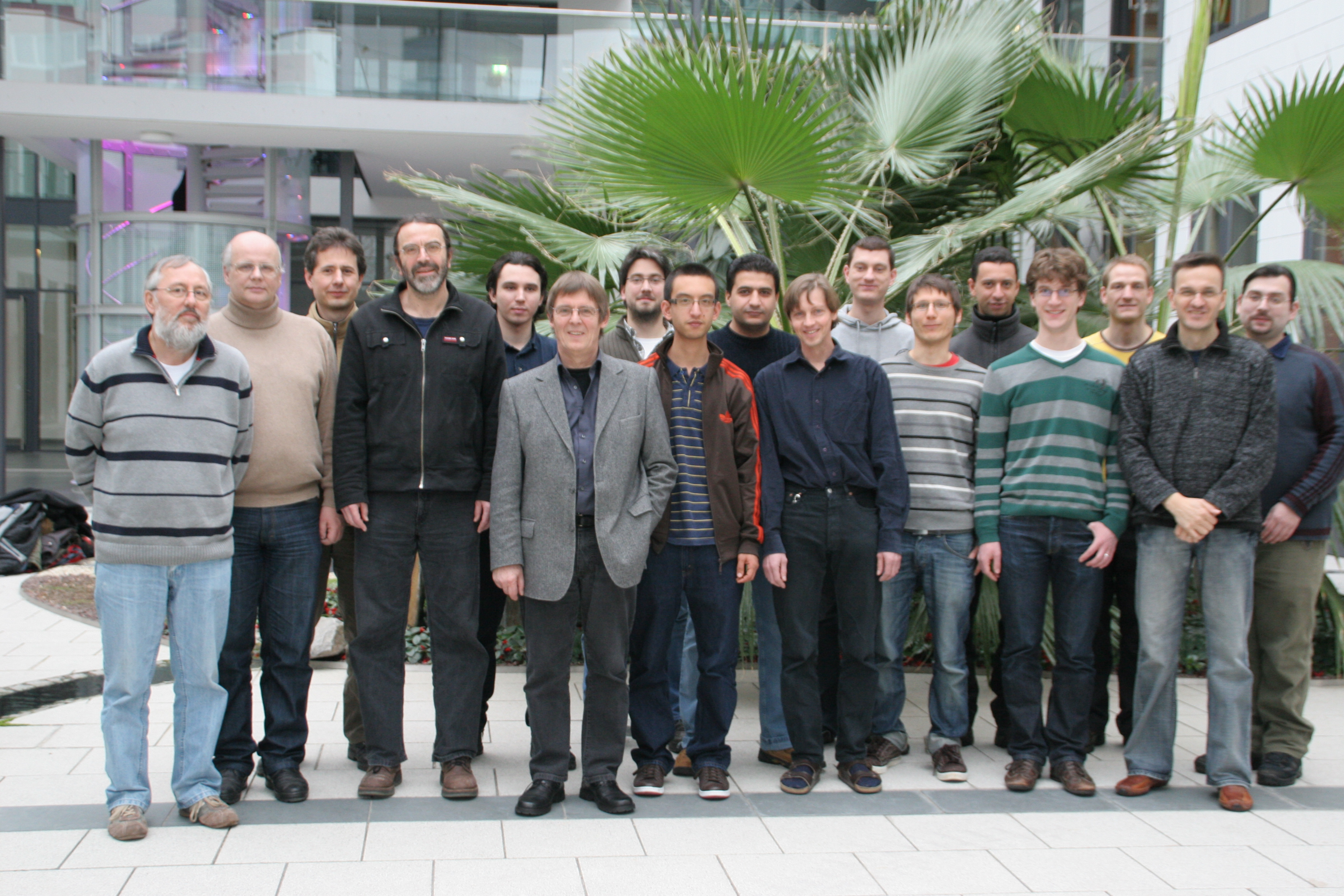}\\
{\sc Singular}-team with Pfister, Sch\"onemann, Lossen, Decker, Greuel, ...
\end{center}

\end{document}